\documentclass[article]{ajs}

\author{Stefanos Leonardos\\ Singapore University of Technology \\and Design \And Costis Melolidakis\\ National and Kapodistrian University\\ of Athens}
\title{On the Mean Residual Life of Cantor-Type Distributions: Properties and Economic Applications}

\Plainauthor{Stefanos Leonardos, Costis Melolidakis} 
\Plaintitle{On the Mean Residual Life of Cantor-Type Distributions: Properties and Economic Applications} 
\Shorttitle{Mean Residual Life of Cantor-Type Distributions} 

\Abstract{
In this paper, we consider the \emph{mean residual life} (MRL) function of the Cantor distribution and study its properties. We show that the MRL function is continuous at all points, locally decreasing at all points outside the Cantor set and has a unique fixed point which we explicitly determine. These properties readily extend to the parametric family of $p$-singular, Cantor type distributions introduced by \cite{Man83}. The findings offer evidence that, contrary to common perceptions, Cantor-type distributions are tractable enough to be considered for practical applications. We provide such an example from the field of economics in which Cantor-type distributions can be used to model markets with recurrent bandwagon effects and show that earlier anticipated bandwagon effects lead to higher monopolistic prices. We conclude with a simple implementation of the algorithm by \cite{Ch91} to plot Cantor-type distributions.}
\Plainkeywords{Cantor distribution, mean residual life function, decreasing generalized mean residual life function, fixed points} 

\Volume{50}
\Month{July}
\Year{2021}
\Pages{65--77}

\Address{
  Stefanos Leonardos\\
  Engineering Systems and Design\\
  Singapore University of Technology and Design\\
  8 Somapah Rd, 487372 Singapore, Singapore\\
  E-mail: \email{stefanos\_leonardos@sutd.edu.sg}\\
  URL: \url{https://stefanosleonardos.com}\\

  Costis Melolidakis\\
  Department of Mathematics\\
  National and Kapodistrian University of Athens\\
  Panepistimioupolis, 15784 Athens, Greece\\
  E-mail: \email{cmelol@math.uoa.gr}
}

\usepackage{amssymb,amsmath,amsthm,cleveref,enumitem}
\usepackage[usenames,dvipsnames]{xcolor}
\newcommand{\ex}{\mathbb{E}}
\newcommand{\F}{\bar{F}}

\newcommand{\m}{\mathrm{m}}

\newcommand{\e}{\mathrm{\ell}}

\newcommand{\du}{du}
\renewcommand{\F}{\bar{F}}
\renewcommand{\(}{\left(}
\renewcommand{\)}{\right)}
\newcommand{\C}{\mathcal C}

\newtheorem{theorem}{Theorem}[section] 
\newtheorem{lemma}[theorem]{Lemma} 
\theoremstyle{definition}
\newtheorem{application}[theorem]{Application} 
\newtheorem{remark}[theorem]{Remark} 

\usepackage[framed,numbered,autolinebreaks,useliterate]{mcode}
\begin{document}

\section{Introduction}
\label{sec:introduction}
Theoretical study of complex social and economic phenomena entails a trade-off between tractability and robustness of the model to perturbations of its working assumptions. This trade-off is particularly evident in the study of decision making under uncertainty: to optimize the outcome of their actions, economic agents and social planners need to accurately reason about parameters for which they only have scarce information at the time of the decision, such as economic growth or retail demand. In many cases, the quest for simplicity (and hence, tractability) of the selected theoretical model leads to the adoption of simplifying assumptions which in turn, compromise the robustness of the resulting decision rules in practice.\par
In this context, singular distributions, although naturally emerging even in low-dimensional economic settings \cite{Mac97,Mit04, Mit06}, are typically avoided due to their perceived intractability. While the use of absolutely continuous distributions may deliver accurate predictions in multiple settings, see e.g., \cite{La06,Ba13}, in other cases, the lack of accuracy in the modelling of uncertainty leads to conservative pricing strategies due to the consideration of worst-case scenarios or more generally to significantly suboptimal decisions \cite{Ber11,Co16}.\par
In the present paper, motivated by the recurrent theme of monopoly pricing under demand uncertainty in markets with linear stochastic demand, we study the properties of the \emph{mean residual life} (MRL) function $m\(x\), x\in [0,1]$ of the Cantor distribution (cf. equation \eqref{mrl}). This setting complements a larger study that is presented in \cite{Le18} in which we derive mild sufficient conditions on the demand distribution --- or more accurately on the seller's belief about the demand --- that ensure a unique optimal price.\footnote{Other related works include \cite{Leon20,Leo20,Leon21}, preliminary versions of which appear in \cite{Bel18,Kok18}.} The resulting class of distributions is particularly inclusive since it only requires that the cumulative distribution function (CDF) $F\(x\)$ is continuous, that the distribution has finite second moment and that the term $m\(x\)/x$ is decreasing. Optimal prices correspond to fixed points of the MRL function and these conditions are enough to ensure existence and uniqueness of a fixed point. \par

Interestingly, some mild regularity conditions --- continuity of the CDF and finiteness of moments --- are also satisfied by the singular Cantor distribution \cite{Do06}. This provides an incentive to test the applicability of our derived pricing rule under minimum possible assumptions. Specifically, if we only require that the CDF $F\(x\)$ is continuous over a finite interval, while $m\(x\)/x$ may not be decreasing, then can we still have existence of a unique optimal price, i.e., of a unique fixed point of the MRL function? The present paper addresses precisely this question and establishes that this indeed true in the context of the Cantor distribution. More specifically, our findings are the following.
\begin{itemize}[noitemsep,leftmargin=0.4cm]
\item We leverage a result of \cite{Do06} to show that the MRL function of the Cantor distribution is continuous at all points and locally decreasing at all points outside the Cantor set. We then proceed to show that this implies the existence of a unique fixed point of the MRL function, i.e., of a unique solution to the equation $m\(x\)=x, x\in[0,1]$, which we explicitly determine at $x=5/12$. These findings are formalized in \Cref{thm:cantor}. 
\item We then consider a parametric family of $p$-singular Cantor-type distributions originally discussed by \cite{Man83} and \cite{Ch91}. We show that \Cref{thm:cantor} readily extends to the case of arbitrary $p>0$ and provide a closed form expression for the unique fixed point in $[0,1]$ of any such distribution, cf. \Cref{thm:extension}. We complement our results with illustrations of the cumulative distribution functions and MRL functions of the Cantor-type distributions for various values of $p>0$ and provide an (simple) implementation in Matlab R2019b of the algorithm originally suggested by \cite{Ch91} to reproduce the illustrations. 
\end{itemize}
While the above findings provide an affirmative answer to the robustness of our derived pricing rule --- i.e., that existence and uniqueness of a fixed point of the MRL function may occur even under the minimum possible assumption of continuity of the CDF over a closed interval --- they leave an integral question unanswered: are there realistic economic settings which can be modelled via Cantor-type distributions? Perhaps surprisingly, the answer to this question is affirmative. \par
Apart from the natural emergence of Cantor type attractors in supply and demand models and models of economic growth \cite{Ono00, Mit06} and \cite{Ono18}, the exact shape of the MRL functions, cf. \Cref{fig:cantor} and \Cref{fig:type<1,fig:type>1}, conveys to Cantor-type distributions an interesting property in the context of modelling demand uncertainty. Specifically, the kind of non-monotonic behavior that they exhibit may be leveraged to model markets with positive demand shocks that are expected to occur whenever the demand reaches or exceeds some threshold or equivalently \emph{bandwagon effects}, i.e., situations in which demand leads to more demand \cite{Her09}. Such shocks may be due to successful introduction of the product to new market segments following its success in its currently operating markets or adoption of the product by new groups of consumers following its increased visibility via its current consumers. Beyond product or service consumption, bandwagon effects are also widespread in social and political phenomena that include voting preferences and diffusion of ideas \cite{Bar19}.\par
To demonstrate the potential use of the parametric family of $p$-singular Cantor-type distributions in the modelling and study of this kind of markets, we present an application --- monopoly pricing under demand uncertainty. Based on the shape of the MRL function, the $p$ parameter can be seen to control the location and intensity of the positive demand shocks. The existence of a unique fixed point of the MRL function leads to a unique optimal price described by the closed form expression in \Cref{thm:extension} and provides useful insight about monopolistic pricing in such markets. In particular, anticipation of bandwagon effects for lower demand values leads to higher monopolistic prices which in turn, are known to produce negative effects on consumer and social welfare. \par
In a nutshell, while our results may be of independent theoretical interest in the study of singular distributions, they also suggest that singular distributions can be more useful than commonly perceived in the modelling and study of social and economic phenomena.

\section{The mean residual life function of the Cantor distribution}
\label{sec:main}

For $C_0=[0,1]$ and $C_n=\frac{C_{n-1}}{3}\cup\(\frac23+\frac{C_{n-1}}{3}\)$ for $n\ge 1$, the Cantor set $\C$ is defined as $\C:=\bigcap_{n=1}^{\infty} C_n$. The Cantor distribution is defined as the distribution that is uniform on the Cantor set \citep{Gu13}. Let $F$ denote its cumulative distribution function and let $X\sim F$ denote a random variable with distribution $F$. $F$ is continuous at every $x\in[0,1]$ and $F'\(x\)=0$ almost everywhere which means that $F$ is a \emph{continuous singular} distribution. Also, let  
\begin{equation}\label{mrl}\m\(x\):=\begin{cases}\ex\(X-x\mid X>x\)=\displaystyle\frac{1}{\vphantom{\tilde{\F}}\F\(x\)}\int_{x}^{1}\F\(u\)\du, & \text{if } x<1\\0, & \text{otherwise}\end{cases}
\end{equation}
denote the \emph{mean residual life} (MRL) function of $X$ \citep{Sh07,Lax06}. Fixed points of the MRL function, i.e., solutions of the equation $\m\(x\)=x$, are related to optimal solutions of optimization problems, cf. \Cref{app:application}. When studying solutions of the fixed point equation, it is sometimes useful to study directly the function $\m\(x\)/x$. To that end, let $\e\(x\):=\m\(x\)/x$, for $0<x<1$, denote the \emph{generalized mean residual life (GMRL)} function of $X$ \citep{Le18}. Since $X$ is non-negative, we have that $\m\(0\)=\ex X=1/2$, see e.g., \cite{Gu13}. $X$ has the \emph{decreasing mean residual life} (DMRL) property, if $\m\(x\)$ is non-increasing for $x<1$ and similarly, $X$ has the \emph{decreasing generalized mean residual life} (DGMRL) property, if $\e\(x\)$ is non-increasing for $x<1$. Finally, we will say that a continuous function $f : \(0, 1\) \to \mathbb R$ is \textit{locally monotone} at $x \in \(0, 1\)$ if there is an open neighborhood $U$ of $x$ such that $f|_U$ is monotone.

\begin{theorem}\label{thm:cantor}
Let $\m\(x\)$ for $x\in[0,1]$ denote the MRL function of the Cantor distribution. Then 
\begin{enumerate}[label=$\(\roman*\)\;$, noitemsep]
\item $\m\(x\)$ is continuous with $\m\(x\)=\frac{1}{F\(1-x\)}\int_{0}^{1-x}F\(u\)\du$, for all $x\in [0,1]$,
\item both $\m\(x\)$ and $\e\(x\)$ are locally decreasing at a point $x\in \(0,1\)$ if and only if $x\in \left[0,1\right]\setminus C$.
\item the unique solution of the fixed point equation $\m\(x\)=x$ in $\left[0, 1\right]$ is $x^*=\frac{5}{12}$.
\end{enumerate}
\end{theorem}
The statements of \Cref{thm:cantor} are illustrated graphically in \Cref{fig:cantor}.

\begin{figure}[!htb]
\centering
\begin{minipage}{.5\textwidth}
  \centering
  \includegraphics[width=\linewidth]{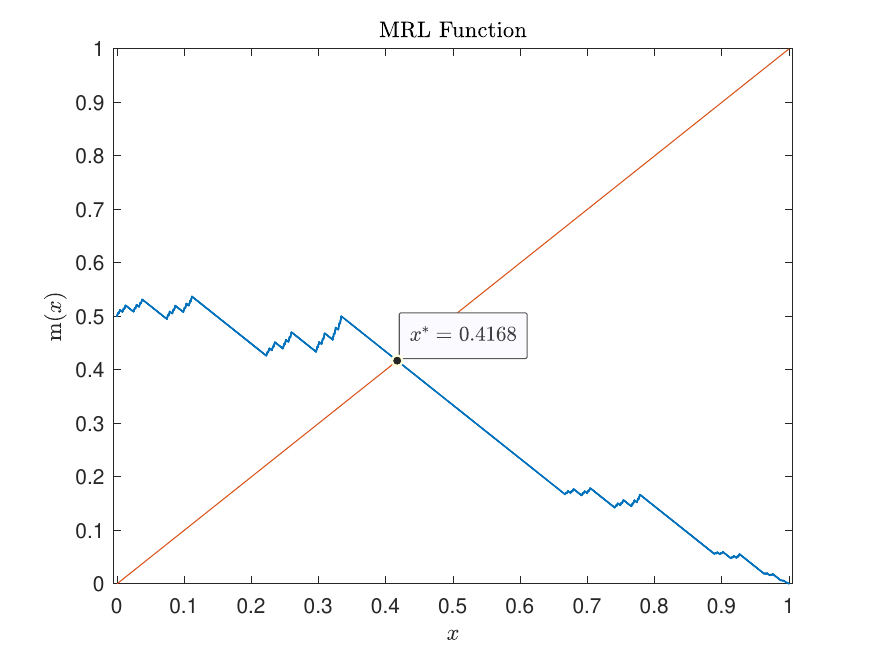}
\end{minipage}%
\begin{minipage}{.5\textwidth}
  \centering
  \includegraphics[width=\linewidth]{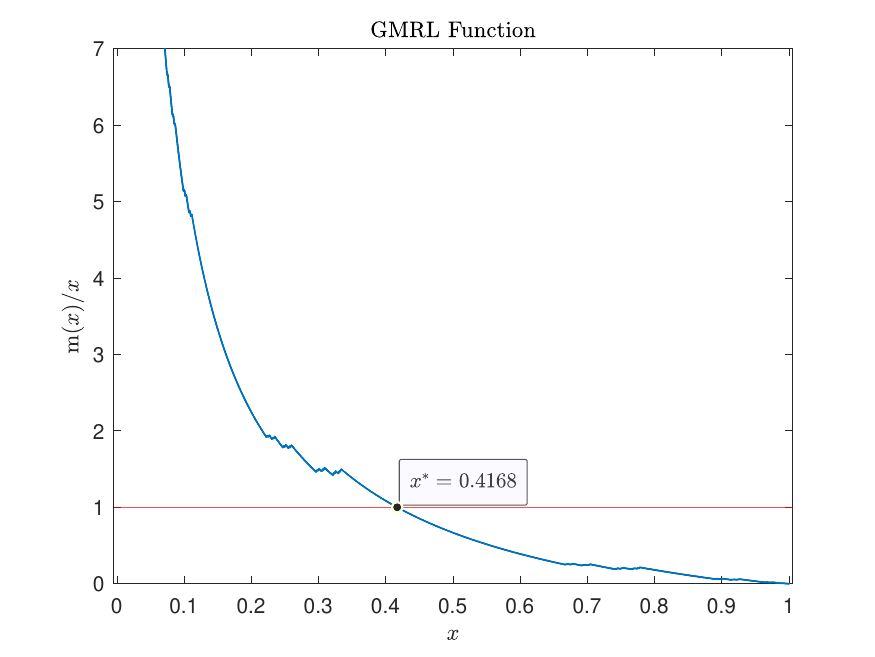}
\end{minipage}
\caption{Cantor distribution: the left panel shows the MRL function $\m\(x\)$ and the right panel shows the GMRL function $\e\(x\)=\m\(x\)/x$ for values away from $0$ where the GMRL grows to infinity. The unique fixed point of $\m\(x\)$ at $x^*=5/12$ is shown in both panels. }
\label{fig:cantor}
\end{figure}

\begin{proof}[Proof of \Cref{thm:cantor}] Using that $\F\(x\)=F\(1-x\)$, for all $x\in [0,1]$, see, e.g., \cite{Ch91}, and by a change of variable in the integration, we obtain (i): \[\m\(x\)=\frac{1}{\F\(x\)}\int_{x}^{1}\F\(u\)\du=\frac{1}{F\(1-x\)}\int_{x}^{1}F\(1-u\)\du=\frac{1}{F\(1-x\)}\int_{0}^{1-x}F\(u\)\du\] The continuity of $\m\(x\)$ is immediate, since $\m\(x\)$ is the quotient of two continuous functions (with the denominator being throughout non-zero) with $\lim_{x\to1-}\m\(x\)=0$ by \cite{Ha81}, Proposition 2(d). Statement (ii) follows from Theorem 3.7 of \cite{Do06} on $f\(x\):=\frac{1}{F\(1-x\)}$ and $\phi\(x\):=\int_{0}^{1-x}F\(u\)\du$. Indeed, $f: \(0,1\)\to \left[0,+\infty\)$ is increasing and continuous and $\phi: \(0,1\)\to\left[0,+\infty\)$ is strictly decreasing with finite derivative $\phi'\(x\)$ at every $x\in \(0,1\)$. The set of constancy of $f$, $L_f$, is given by $L_f=\left[0,1\right]\setminus C$ and has Lebesgue measure $1$. Hence, by the above Theorem, the product $f\phi$ is locally decreasing at a point $x\in\(0,1\)$ if and only if $x\in L_f$. Since $\m\(x\)=f\(x\)\phi\(x\)$, the result follows. Taking $\phi\(x\):=\frac1x\int_{0}^{1-x}F\(u\)\du$, establishes the result for $\e\(x\)$. Statement (iii) of \Cref{thm:cantor} can be derived by applying \Cref{lem:cantor}--(i) on $0\le y\le 1/3$ and \Cref{lem:cantor}--(ii) on $2/3\le y < 1$. Specifically, using the symmetry of $F$, it is a straightforward calculation to verify the values of $m\(x\)$ in the following steps:
\begin{enumerate}[itemindent=0.5cm,label=step \arabic*:]
\item for $y=0, \delta=1/4$, since $\m\(0\)=\ex\(\alpha\)=1/2$, we get $\m\(x\)-x>0$ for all $0\le x<1/4$.
\item for $y=\frac{20}{81}, \delta=\frac{343}{3564}$, since $\m\(20/81\)=\frac{29}{66}$, we get $\m\(x\)-x>0$ for all $\frac{20}{81}\le x<\frac{1223}{3564}$.
\end{enumerate}
Combining steps 1. and 2., we obtain that $\m\(x\)-x>0$ for all $x \in [0,1/3]$. For $x \in \left[1/3,2/3\right]$, we write $x=1/3+t$ with $0\le t\le 1/3$ and by \Cref{lem:cantor}--(iii) 
\[\m\(1/3+t\)=\m\(1/3\)-t=\frac{1}{F\(2/3\)}\int_{0}^{2/3}F\(u\)du=1/2-t.\] Hence, the fixed point equation $\m\(1/3+t\)=1/3+t$ for $0\le t\le 1/3$ is equivalent to $1/2-t=1/3+t$ and therefore its only solution is given by $t=1/12$. This shows that $x^*=1/3+1/12=5/12$ is the only solution of the fixed point equation $x^*=\m\(x^*\)$ in $\left[1/3,2/3\right]$. For $2/3< x<1$ we apply \Cref{lem:cantor}--(ii) for $y=1-\epsilon$  with $\epsilon>0$ arbitrary and $\delta=1/3$, which gives that for all $2/3<x<1-\epsilon$, $\m\(x\)-x<\m\(1-\epsilon\)-1+2\cdot\(1/3\)=\m\(1-\epsilon\)-\(1/3\)<0$, since $\lim_{x\to1-}\m\(x\)=0$. Since $\epsilon>0$ was arbitrary this concludes the proof of \Cref{thm:cantor}--(iii).
\end{proof}

The tractability of the Cantor distribution mainly stems from the continuity of its cumulative distribution function which is inherited by its MRL function. Based on \Cref{thm:cantor}, the Cantor distribution does not satisfy the DMRL nor the DGMRL property, yet its MRL function has a unique fixed point. \\

\section{Extension: parametric family of Cantor-type distributions}
\label{sec:extension}
\cite{Ch91} studies the parametric family of \emph{$p$-singular Cantor-type} random variables $X_p$ for $p>0$ with cumulative distribution functions $F_p$ characterised as the unique real-valued, monotone increasing functions that satisfy the conditions
\begin{enumerate}[label=(\roman*)]
\item $F_p\(x/3\)=F_p\(x\)/\(p+1\)$, for $x\in[0, 1]$,
\item $F_p\(1-x\)=1-pF_p\(x\)$, for $x\in[0, 1/3]$.
\end{enumerate}
Referring to a discussion by \cite{Man83}, \cite{Ch91} notes that the corresponding probability measures are invariant under the \emph{inverted V} transformation. In the defining conditions, he also requires that $F_p\(0\)=0$ which follows directly from (i). Moreover, from (i) and (ii) it follows that $F_p\(1/3\)=F_p\(2/3\)=1/\(p+1\)$ and hence, that $F_p\(x\)=1/\(p+1\)$ (due to monotonicity) for all $x\in[1/3,2/3]$. Accordingly, condition (ii) extends to any $x\in[0,2/3]$. The Cantor distribution is the unique real valued, monotone increasing function that corresponds to $p=1$, in which case condition (ii) holds for any $x\in[0,1]$, cf. proof of \Cref{thm:cantor}. \par
In this section, we show that the statements of \Cref{thm:cantor} readily generalize to any $p>0$. In particular, the MRL function $\m_p\(x\)$ of any $p$-singular Cantor type distribution has a unique fixed point which can be explicitly calculated as a function of $p$. Such fixed points are decreasing in $p$ and lie within $\(3/8,1/2\)$ with the lower and upper limits approximated for $p\to+\infty$ and $p\to0^+$ respectively. These findings are formalized in \Cref{thm:extension}.

\begin{theorem}\label{thm:extension}
Let $F_p\(x\)$ and $\m_p\(x\)$ for $x\in[0,1]$ denote the cumulative distribution function (CDF) and mean residual life (MRL) function, respectively, of the $p$-singular Cantor-type random variable $X_p$. Then, for each $p>0$ it holds that
\begin{enumerate}[label=(\roman*)]
\item $\mathbb E[X_p]=\frac{3}{2}\cdot\frac{p}{2p+1}$,
\item $\m_p\(x\)$ and $\e_p\(x\)$ are continuous on $[0,1]$ and locally decreasing at a point $x\in \(0,1\)$ if and only if $x\in \left[0,1\right]\setminus C$, and
\item the fixed point equation $\m_p\(x\)=x$ has a unique solution $x^*_p$ in $[0,1]$ given by 
\begin{equation}\label{eq:extension}
x_p^*=\frac16+\frac{5p+4}{12\(2p+1\)}\, .
\end{equation}
\end{enumerate}
In particular, the fixed point $x^*_p$ is decreasing in $p$ and takes values in $\(3/8,1/2\)$.
\end{theorem}
\begin{remark}\label{rem:extension}
The proof of \Cref{thm:extension} follows the proof of \Cref{thm:cantor} and is deferred to \Cref{app:extension}. For $p=1$, equation \eqref{eq:extension} recovers the result of \Cref{thm:cantor}, namely that $x^*_1=5/12$.\par
\cite{Ch91} presents a simple algorithm to plot the CD and MRL functions of the $p$-singular Cantor type distributions. Based on a simple implementation of this algorithm in Matlab R2019 which allows for multiple iterations before reaching memory capacity, we have produced the following plots of the CD and MRL functions for various values of $p \in \(0,1\)$ and $p\in\(0,+\infty\)$ in \Cref{fig:type<1,fig:type>1}, respectively. The algorithm can be found in \Cref{app:algorithm}. The MRL plot for $p=0.01$, cf. \Cref{fig:type<1}, and for lower values of $p$ not depicted here, indicates that a second fixed point may occur at $x=1/3$ in the limit as $p\to0^+$. However, since $\m_p\(1/3\)=\frac{5p+4}{6\(2p+1\)}>1/3$ for any $p>0$ (cf. proof of \Cref{thm:extension} in \Cref{app:extension}), $x_p^*$ remains the unique fixed point in the whole $[0,1]$ interval for any $p>0$.
\end{remark}

\begin{figure}[!htb]
\centering
\begin{minipage}[b]{.5\textwidth}
  \centering
  \includegraphics[width=\linewidth]{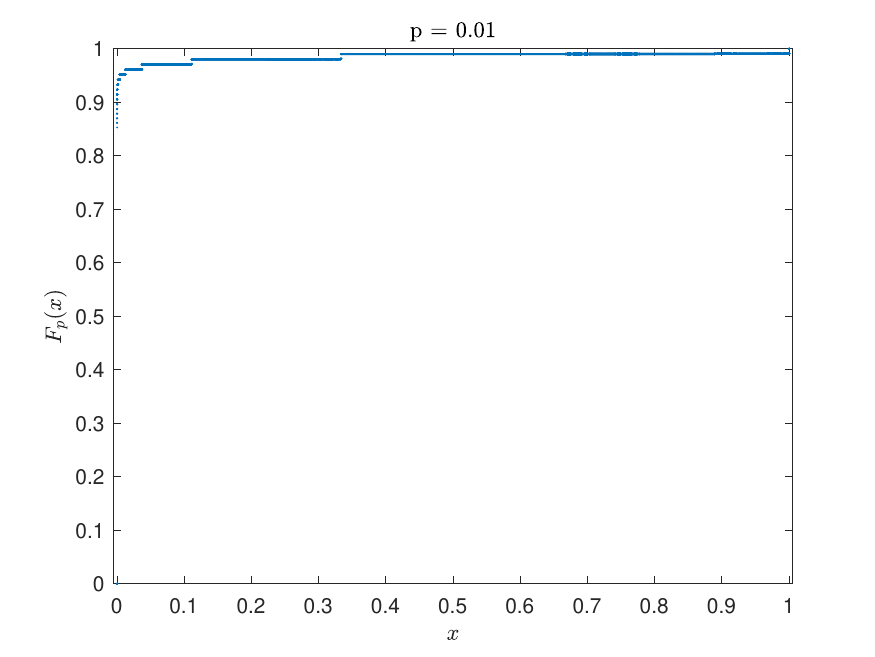}
\end{minipage}
\begin{minipage}[b]{.5\textwidth}
  \centering
  \includegraphics[width=\linewidth]{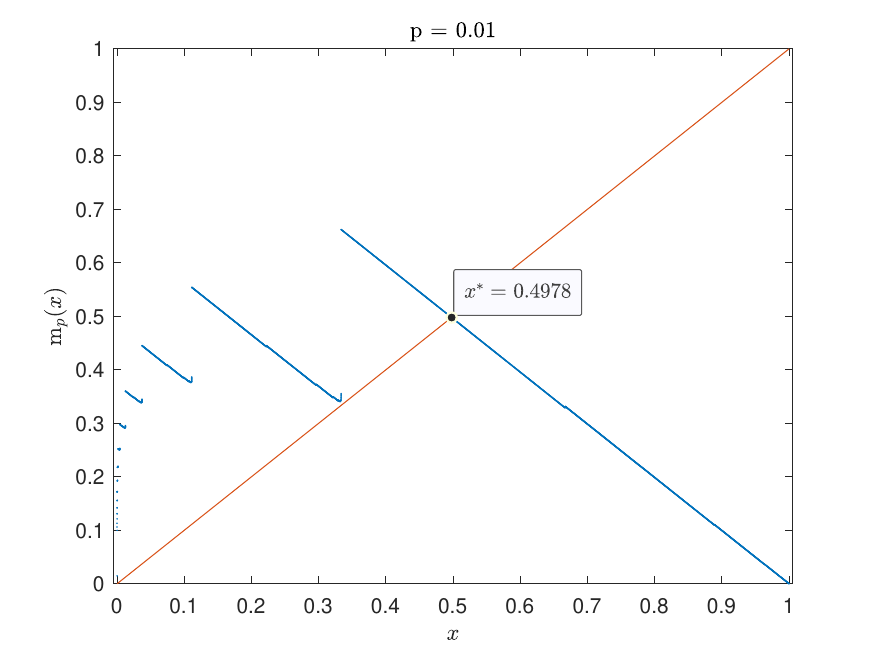}
\end{minipage}
\begin{minipage}[b]{.5\textwidth}
  \centering
  \includegraphics[width=\linewidth]{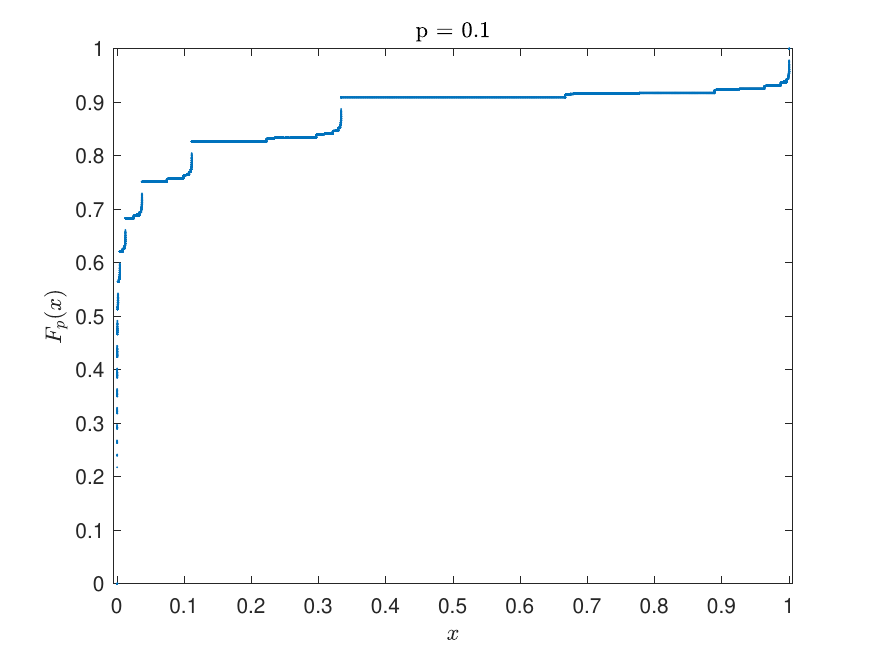}
\end{minipage}
\begin{minipage}[b]{.5\textwidth}
  \centering
  \includegraphics[width=\linewidth]{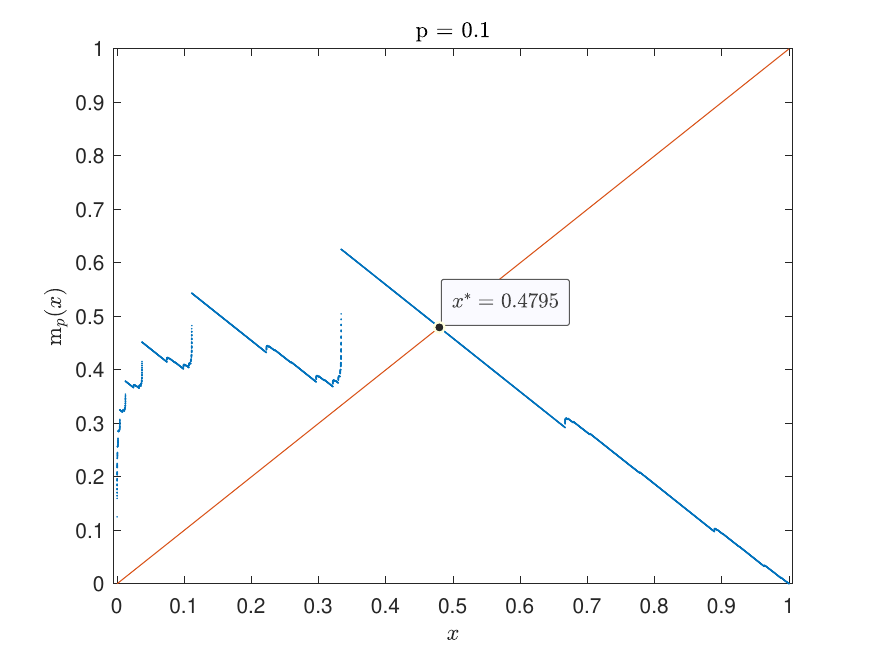}
\end{minipage}
\begin{minipage}[b]{.5\textwidth}
  \centering
  \includegraphics[width=\linewidth]{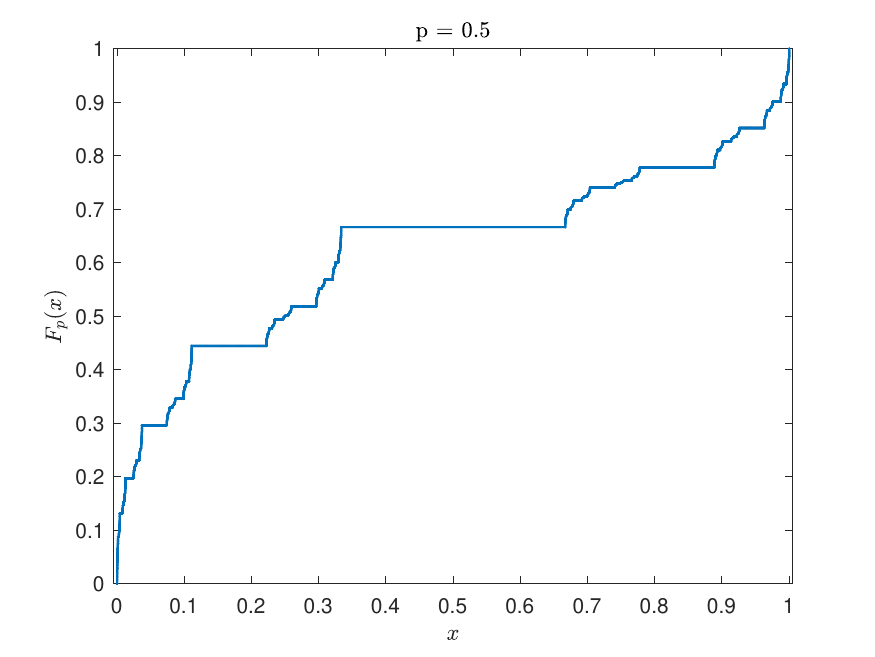}
\end{minipage}
\begin{minipage}[b]{.5\textwidth}
  \centering
  \includegraphics[width=\linewidth]{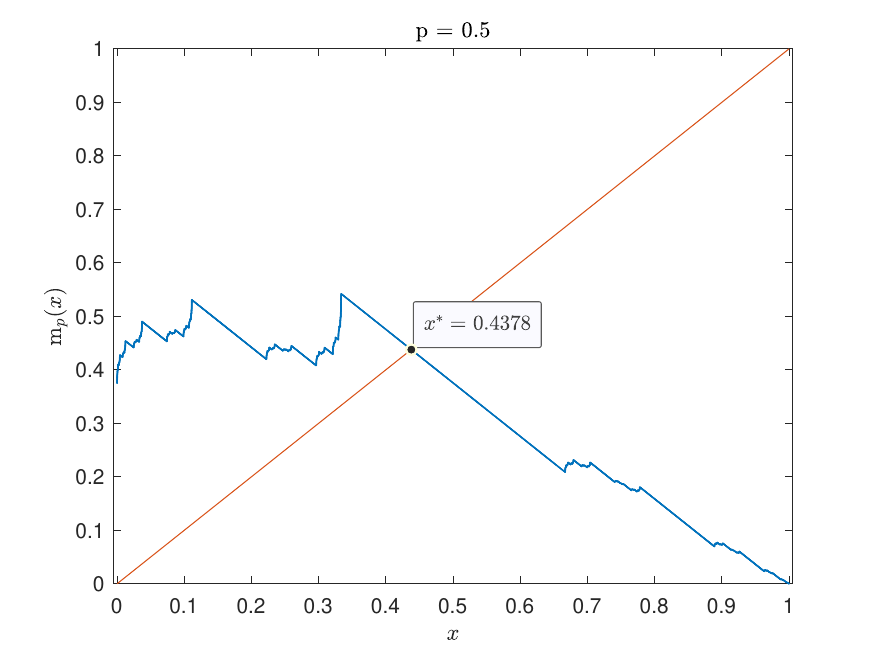}
\end{minipage}
\caption{Cumulative distribution functions $F_p\(x\)$ (left column) and mean residual life functions $\m_p\(x\)$ with highlight of the single fixed point (right column) of the parametric family of $p$-singular distributions for various values of $p<1$. The unique fixed point is decreasing in $p$ and takes values close to $0.5$ for values of $p$ approximating $0^+$. The plots have been produced using 1000 initial points and 17 iterations of the algorithm of \cite{Ch91} as described in \Cref{app:algorithm}.}
\label{fig:type<1}
\end{figure}

\begin{figure}[!htb]
\begin{minipage}[b]{.5\textwidth}
  \centering
  \includegraphics[width=\linewidth]{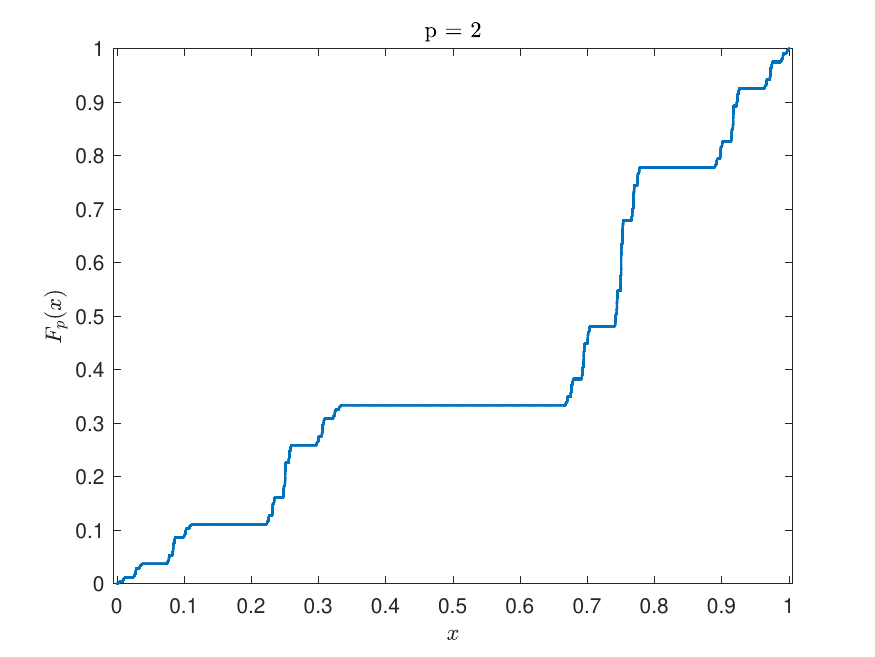}
\end{minipage}
\begin{minipage}[b]{.5\textwidth}
  \centering
  \includegraphics[width=\linewidth]{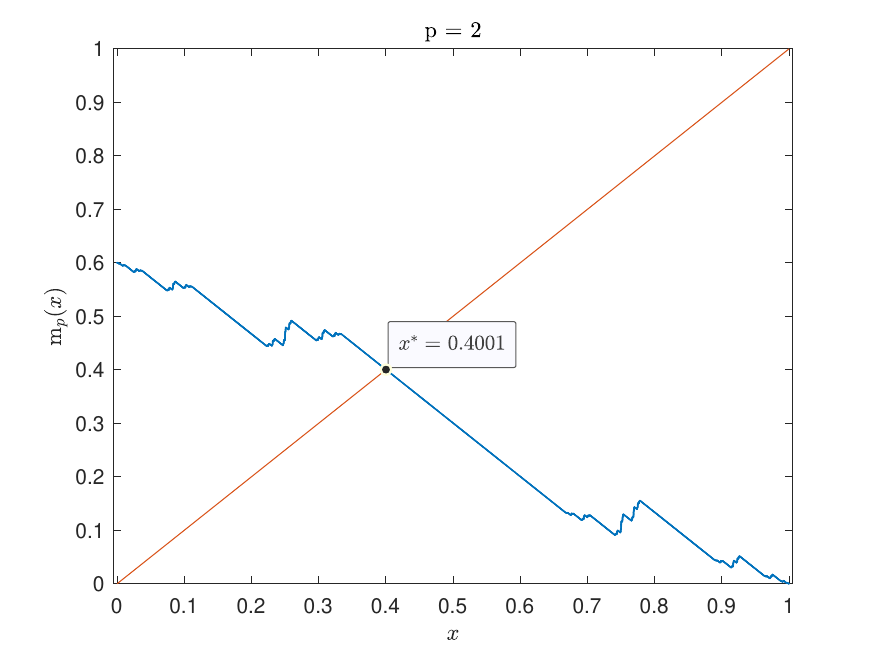}
\end{minipage}
\begin{minipage}[b]{.5\textwidth}
  \centering
  \includegraphics[width=\linewidth]{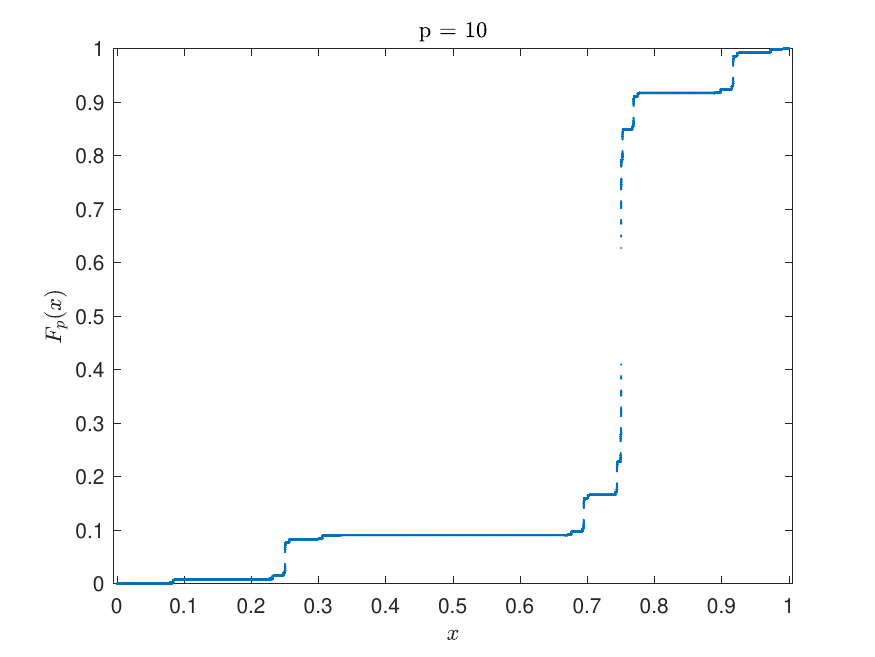}
\end{minipage}
\begin{minipage}[b]{.5\textwidth}
  \centering
  \includegraphics[width=\linewidth]{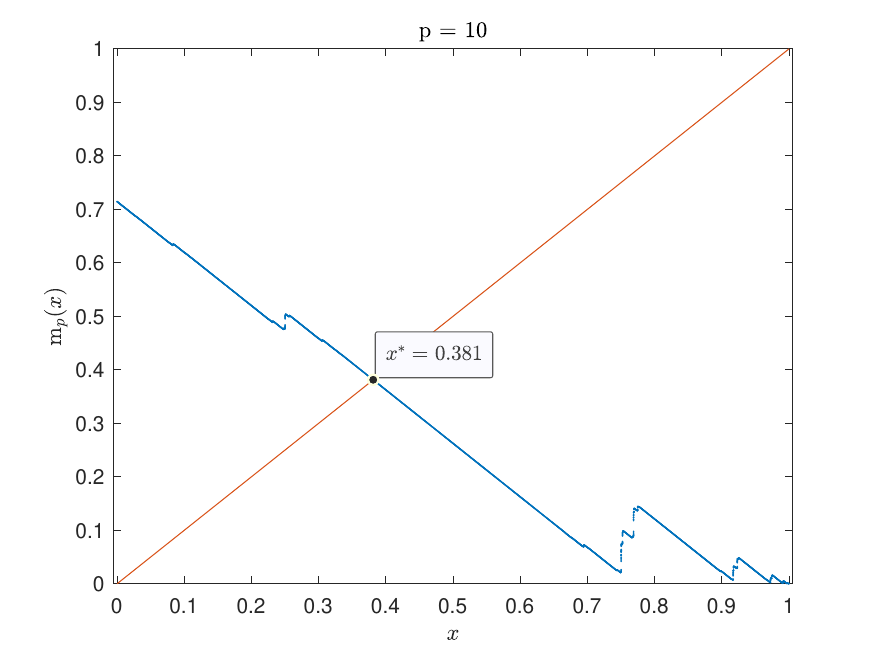}
\end{minipage}
\begin{minipage}[b]{.5\textwidth}
  \centering
  \includegraphics[width=\linewidth]{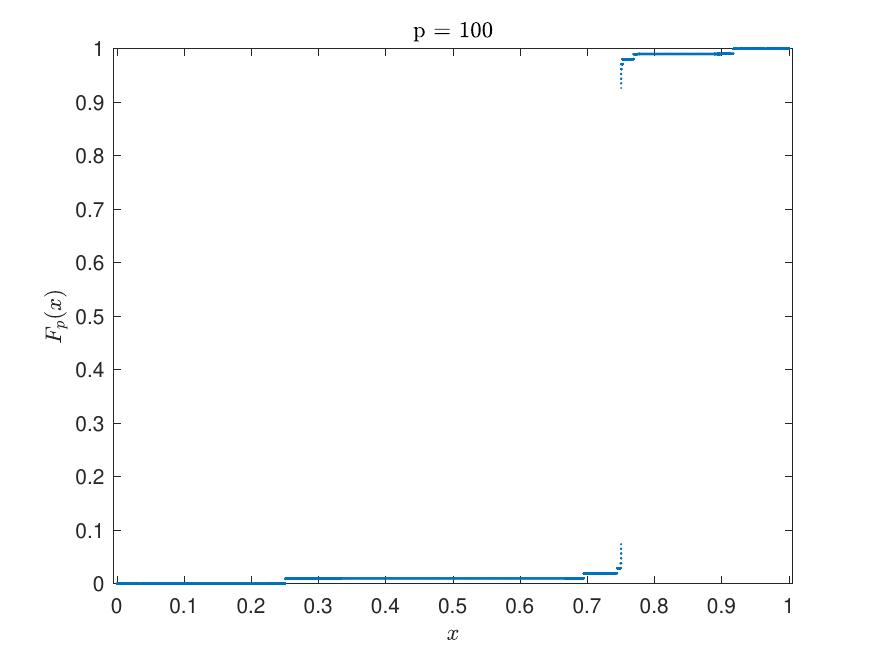}
\end{minipage}
\begin{minipage}[b]{.5\textwidth}
  \centering
  \includegraphics[width=\linewidth]{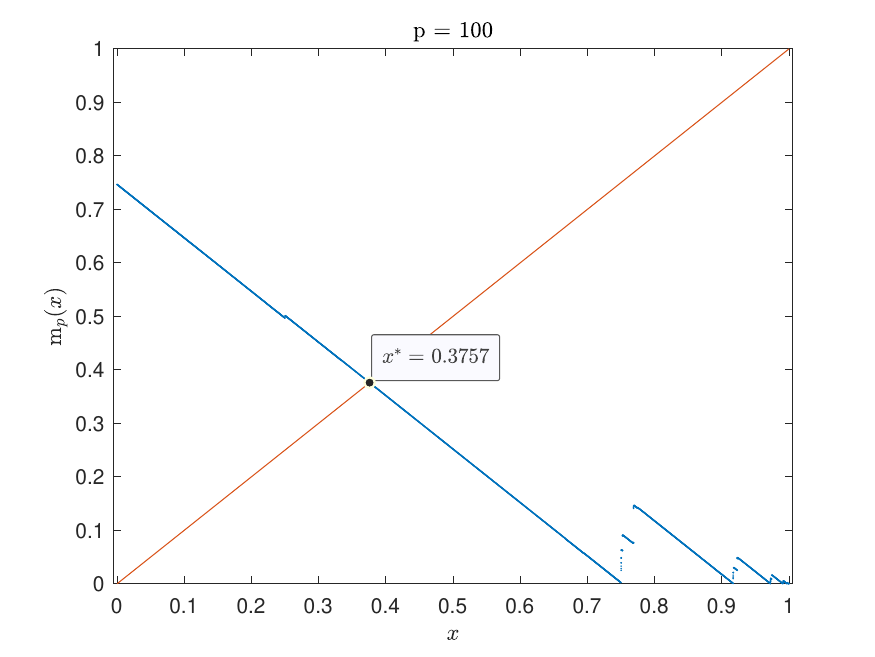}
\end{minipage}
\caption{Cumulative distribution functions $F_p\(x\)$ and mean residual life functions $\m_p\(x\)$ of the parametric family of $p$-singular with highlight of the single fixed point (right column) of the parametric family of $p$-singular distributions for various values of $p>1$. The unique fixed point is decreasing in $p$ and takes values close to $0.375$ for values of $p$ approximating $+\infty$. The plots have been produced using 1000 initial points and 17 iterations of the algorithm of \cite{Ch91} as described in \Cref{app:algorithm}.}
\label{fig:type>1}
\end{figure}

\section{An application: pricing in markets with bandwagon effects}
The above derived properties of the MRL functions --- their local monotonicity and the uniqueness of a fixed point --- render Cantor-type distributions more relevant for practical applications than generally thought. In the context of modelling demand uncertainty, the MRL function describes the mean residual demand for each demand level \cite{Le18}.\par
The shape of the MRL function of $p$-singular Cantor distributions, cf. \Cref{fig:type<1,fig:type>1}, can be utilized to model markets with recurrent \emph{bandwagon effects} \cite{Her09,Bar19}. Such markets are characterized by positive demand shocks whenever demand reaches (or exceeds) some threshold. For instance, if a product is adopted by a certain number of consumers (e.g., a cluster in the population), then the expected residual demand may increase due to increased visibility of the product which attracts more consumers. The illustrations in \Cref{fig:type<1,fig:type>1} indicate that the location and intensity of these effects --- that correspond to continuous jumps of the MRL function (the visible discontinuities in the plots are due to the finite number of executions of the plotting algorithm) --- are calibrated by the selection of parameter $p$: lower values of $p$ correspond to more intense effects that appear at lower demand values, whereas larger values of $p$ smoothen the effects and move them towards the upper end of possible demand values. The original Cantor distribution, $p=1$, corresponds to a symmetric case, cf. \Cref{fig:cantor}.\par  
From the perspective of a monopolistic seller who is facing a market with this kind of effects, the unique fixed point of the MRL function of Cantor-type distributions turns out to be particularly useful in the design of a unique optimal pricing policy. In turn, and from a theoretical perspective, a unique optimal price allows for accurate predictions and meaningful comparative statics analysis (study of the response of the optimal price to market parameters). The details of modelling such an application are given in \Cref{app:application} (see \cite{Le18} for the general statement of this problem).
\begin{application}\label{app:application}
A seller is selling goods to a market with linear demand $q\(x\)=X-x$ where $x$ denotes the price, $q\(x\)$ the quantity that is demanded at price $x$ and $X$ the demand level. $X$ is uncertain at the time of the sellers' pricing decision and distributed according to a $p$-singular Cantor-type distribution with CDF $F_p$. Parameter $p$ is selected according to the intensity and location of the anticipated positive demand shocks (bandwagon effects). The seller is concerned to determine the optimal price $x^*$ that will maximize their expected payoff $\Pi\(x\):=x\ex \(X-x\)_+$ where $\(X-x\)_+$ is the positive part of $\(X-x\)$, i.e., $\(X-x\)_+=X-x$ if $X\ge x$ and is $0$ otherwise. As shown in \cite{Le18}, the set of optimal prices coincides with the set of fixed points of the MRL function, i.e., any optimal price $x^*$ satisfies $\m\(x^*\)=x^*$.  By \Cref{thm:extension}-(iii), $x_p^*$ is unique for any $p>0$ and equal to $x_p^*=\frac16+\frac{5p+4}{12(2p+1)}$. Since $x_p^*$ is decreasing in $p$, the seller charges a higher price in a market with earlier anticipated bandwagon effects, i.e., bandwagon effects that occur at lower demand level. Hence, the expectation of such effects has an adverse impact on both consumer and social welfare which are known to diminish as prices increase.
\end{application}

\section{Conclusions}
\label{sec:conclusions}
In this paper, we studied the mean residual life (MRL) functions of the Cantor and the $p$-singular, Cantor-type distributions as defined by \cite{Man83} and \cite{Ch91} and showed that for any $p>0$, the MRL function is continuous, locally decreasing at all points outside the Cantor sets and has a unique fixed point which can be explicitly determined as a decreasing function of parameter $p$. Putting these findings in the context of a larger study on optimal pricing with demand uncertainty, cf. \cite{Le18}, we provided an application to indicate that these patterns and properties of the MRL functions may render Cantor-type distributions useful in the modelling of uncertain demand in markets with bandwagon effects. Parameter $p$ calibrates the intensity and location of the effects, i.e., the demand threshold at which the positive demand shocks occur, whereas the unique fixed point of the MRL function corresponds to a monopolistic seller's optimal price. Accordingly, monopolistic prices are higher in markets with earlier anticipated and more intense bandwagon effects, cf. \Cref{fig:type<1,fig:type>1}, since such markets correspond to lower values of $p$. \par
In addition to their independent theoretical interest, these findings offer evidence that singular distributions are more tractable than commonly thought and suggest that they may be used in the context of economic modelling to accurately capture real-life economic markets.

{\bf Acknowledgements.} The authors thank an anonymous referee for the constructive comments that led to a substantial extension of the original paper. Stefanos Leonardos gratefully acknowledges support by a scholarship of the Alexander S. Onassis Public Benefit Foundation. All plots of this paper have been created in Matlab R2019b under an academic license provided by the Singapore University of Technology and Design (SUTD).

\bibliography{cantorbib}
\appendix
\section{Appendix}

\subsection{Omitted Materials: Section 2}
The proof of statement (iii) of \Cref{thm:cantor} relies on \Cref{lem:cantor}. \Cref{lem:cantor} does not make explicit use of the properties of the Cantor distribution and holds for arbitrary distributions with continuous cumulative distribution function on $[0,1]$.

\begin{lemma}\label{lem:cantor}
Let $y\in \left[0,1\right]$, then
\begin{enumerate}[label=$\(\roman*\)\;$,noitemsep]
\item $\m\(x\)-x>\m\(y\)-y-2\delta$, for all $x\in\left[0,1\right]$ and $\delta>0$ such that $y\le x<y+\delta$,
\item $\m\(x\)-x<\m\(y\)-y+2\delta$, for all $x\in\left[0,1\right]$ and $\delta>0$ such that $y-\delta<x\le y$.
\item $\,$If $F$ is constant on $\left[y, y+\delta\right]\subset\left[0,1\right]$, then $\m\(y+t\)=\m\(y\)-t$, for all $0\le t\le \delta$.
\end{enumerate}
\end{lemma}
\begin{proof}
For statement (i) let $y\in [0,1), \delta>0$ and $y\le x<y+\delta$. Then, by the monotonicity of $F$ 
\begin{align}\m\(x\)&=\frac{1}{1-F\(x\)}\int_{x}^{1}\(1-F\(u\)\)\du\ge\frac{1}{1-F\(y\)}\(\int_{y}^{1}\(1-F\(u\)\)\du-\int_{x}^y1-F\(u\)\du\)\nonumber\\&
=\m\(y\)-\frac{1}{1-F\(y\)}\int_{y}^x\(1-F\(u\)\)\du\ge \m\(y\)-\(x-y\)\label{can}
\end{align}
Since $x<y+\delta$, \eqref{can} implies that $\m\(x\)>\m\(y\)-\delta$ and hence, that $\m\(x\)-x>\m\(y\)-x-\delta>\m\(y\)-y-2\delta$. Statement (ii) follows similarly. For statement (iii), it suffices to write $x=y+t$ and observe that \eqref{can} now holds throughout with equality due to $F$ being constant on $\left[y,y+\delta\right]$. 
\end{proof}
\Cref{lem:cantor}-(iii) provides an alternative proof to the ``if part'' of \Cref{thm:cantor}-(ii) since the cumulative distribution function $F$ of the Cantor distribution is piecewise constant on $\left[0,1\right]\setminus C$. 

\subsection{Omitted Materials: Section 3}
\label{app:extension}

\begin{proof}[Proof of \Cref{thm:extension}]
Let $p>0$. To prove statement (i) of \Cref{thm:extension}, we need to estimate the integral $\int_{0}^{1}F_p\(u\)\du$. Using property (i) in the definition of the $p$-singular distributions and by a change of variables in the integration, we obtain that 
\[\int_{0}^1F_p\(u\)\du=\int_{0}^1\(p+1\)F_p\(u/3\)\du=3\(p+1\)\int_{0}^{1/3}F_p\(u\)\du\,.\]
By a change of variables and using property (ii) of the definition, we further obtain that
\[\int_{2/3}^1F_p\(u\)=\int_{0}^{1/3}F_p\(1-u\)\du=\int_{0}^{1/3}1-F_p\(u\)\du=\frac13-p\int_{0}^{1/3}F_p\(u\)\du\,.\]
To proceed, let $I_1:=\int_{0}^{1/3}F_p\(u\)\du$. Then, since $F_p\(x\)=\frac{1}{p+1}$ for any $x\in[1/3,2/3]$ by definition, we combine the previous two equations to obtain 
\[3\(p+1\)I_1=I_1+\frac{1}{3}\cdot\frac{1}{p+1}+\frac13-pI_1\]
which yields that 
\[I_1=\frac16\cdot \frac{p+2}{(p+1)(2p+1)}\]
Hence, 
\[\mathbb E[X_p]=\int_{0}^1\(1-F_p\(u\)\)\du=1-3\(p+1\)I_1=\frac32\cdot\frac{p}{2p+1}\,.\]
Statement (ii) follows in the same manner as in the proof of \Cref{thm:cantor} by noting that the $p$-transformation affects neither the $x$-axis (hence, the interval of constancy of $F_p$ remains the same) nor the continuity of $F$. To prove statement (iii), let $x\in[1/3,2/3]$. Then, by \Cref{lem:cantor}-(iii), it holds that $\m_p\(x\)=\m_p\(1/3\)-\(x-1/3\)$, with 
\begin{align*}
\m_p\(1/3\)&=\frac{1}{1-F_p\(1/3\)}\int_{1/3}^1\(1-F_p\(u\)\)\du\\&
=\frac{1}{1-\frac{1}{p+1}}\(\int_{1/3}^1\du-\int_{1/3}^{2/3}\frac{1}{p+1}\du-\int_{2/3}^1F_p\(u\)\du\)\\
&=\frac{p+1}{p}\(\frac23-\frac{1}{3\(p+1\)}-\frac13+pI_1\)=\frac{5p+4}{6\(2p+1\)}
\end{align*}
where in the last step, we used that $I_1=\frac{1}{6}\cdot\frac{p+2}{\(p+1\)\(2p+1\)}$. Hence, the fixed point equation $\m_p\(x\)=x$ on $x\in[1/3,2/3]$ can be equivalently written as 
\begin{align*}
\m_p\(1/3\)-\(x-1/3\)=x
\end{align*}
which, by substitution of the derived value of $\m_p\(1/3\)$ yields the fixed point $x_p^*$
\[\frac{5p+4}{6\(2p+1\)}+\frac13=2x_p^* \iff x_p^*=\frac{1}{6}+\frac{5p+4}{12\(2p+1\)}\,.\]
To establish uniqueness, we again apply \Cref{lem:cantor}-(ii) at $y=1-\epsilon$ for $\epsilon>0$ arbitrary (as in the proof of \Cref{thm:cantor}) which yields that $\m_p\(x\)<x$ for any $x\in (1/2,1]$, i.e., that there exists no fixed point in $(1/2,1]$. To establish that $\m_p\(x\)>x$ for all $x\in[0,1/3)$, we start with $y=0$ and $\m_p\(0\)=\mathbb E[X_p]$ and again repeatedly apply \Cref{lem:cantor}-(i) (however, this case becomes numerically more complex as $p\to0$.). Finally, as mentioned in \Cref{rem:extension}, $\m_p\(1/3\)=\frac{5p+4}{6\(2p+1\)}>1/3$ for all $p>0$, which shows that there is no fixed point at $x=1/3$ for any $p>0$ and concludes the proof. 
\end{proof}

\subsection{Algorithm Implementation}
\label{app:algorithm}

In this section, we present the main part of the algorithm that we used to produce \Cref{fig:cantor} and \Cref{fig:type<1,fig:type>1}. The algorithm is due to \cite{Ch91} and the proposed implementation utilizes the routine \verb|unique| of Matlab 2019b to avoid an explicit recursion (although the points at the $x$- and $y$-axes are partially recalculated in every iteration) which results in a very simple implementation shown in the first frame (\verb|function cantor_cdf_mrl|). The algorithm starts with the following initialization 
\begin{enumerate}[label=I.]
\item $F\(x\) = 1/2$ on the interval $[1/3,2/3]$ and $F(0) = 0, F(1) = 1$. 
\end{enumerate}
and then repeats the following steps at each stage
\begin{enumerate}[noitemsep, label=R\arabic*.]
\item shrinks the $x$ axis by a factor of $1/3$,
\item shrinks the $y$ axis by a factor of $1/(p+1)$,
\item flips the resultant graph across the line $x = 1/2$, and then again across the line $y = 1/\(p+1\)$.
\end{enumerate}

For the parametric family of $p$-singular Cantor-type distributions, the $x$-axis transformations remain the same for all values of $p$. The MRL and GMRL functions are then constructed by standard calculations. \par
Obviously, to speed up the \verb|for| loop, one may remove the redundant step of keeping $x$ and $F$ at every iteration. This would be indeed sufficient to plot the CDF. However, to obtain the MRL function one would need to store the entire vector of values and this would again increase the computational burden (or memory requirement) of the \verb|for| loop. For completeness, one option to implement the iteration without the redundant calculations is presented in the second frame (\verb|function cantor_cdf|).


\begin{lstlisting}
function cantor_cdf_mrl

% Initialize 
p=1; %p-singular distributions with p=1 for the Cantor distribution
density=500; %customize number of initial points
iterations=20; %customize number of iterations of the algorithm
x=[0 linspace(1/3,2/3,density) 1]; %initial points at the x-axis
F=[0 linspace(1/(p+1),1/(p+1),density) 1]; %initial values of the cdf at the y-axis. 

for k=1:iterations
    %x-axis: shrink the x-axis by 1/3, flip the result across the x=1/2 line and remove duplicate points for faster implementation. This step is not affected by the selection of p.
    [x, ia]=unique([x/3 x 1-x/3]); 
    %y-axis: shrink the y-axis by 1/p+1 and flip across the line y=1/(p+1) (1/2 for the standard Cantor distribution).   
    F=[F/(p+1) F 1-F*(p/(p+1))]; 
    F=F(ia); %index values according to the duplicate free vector x.
end

plot(x,F,'.', 'MarkerSize', 3) %Use plot instead of scatter(x,F) for faster implementation.
\end{lstlisting}

\begin{lstlisting}
function cantor_cdf

% Initialize as above
...
figure;
hold all
for k=1:iterations
    plot(x, F, '.');
    x=[x*u, 1-x*u];
    F=[F*v, 1-p*v*F];
end
\end{lstlisting}

\end{document}